\documentclass[final]{dmtcs-episciences}

\usepackage[utf8]{inputenc}
\usepackage{subfigure}

\usepackage{amsfonts} 
\usepackage{amsmath}
\usepackage{amssymb}
\usepackage{amsthm}
\usepackage{tikz}
\usepackage{ytableau}

\newtheorem{thm}{Theorem}[section]
\newtheorem{lem}[thm]{Lemma}
\newtheorem{prop}[thm]{Proposition}

\newtheorem{cor}[thm]{Corollary}

\numberwithin{equation}{section}

\makeatletter
\newcommand{\D}{\mathcal{D}}
\newcommand{\F}{\mathcal{F}}

\newcommand{\Rmnum}[1]{\expandafter\@slowromancap\romannumeral #1@}
\makeatother

\title{Upward-closed hereditary families in the dominance order}
\author{Michael D. Barrus\affiliationmark{1}
\and Jean A. Guillaume\affiliationmark{2}
}
\affiliation{
University of Rhode Island, United States of America\\
Sacred Heart University, United States of America
}
\keywords{forbidden subgraph, degree sequence, majorization, dominance order}
\received{2019-08-02}

\revised{2021-04-22}

\accepted{2021-11-05}

\publicationdetails{23}{2021}{3}{16}{5666}

\begin{document}
\maketitle

\begin{abstract}
The majorization relation orders the degree sequences of simple graphs into posets called dominance orders. As shown by Ruch and Gutman (1979) and Merris (2002), the degree sequences of threshold and split graphs form upward-closed sets within the dominance orders they belong to, i.e., any degree sequence majorizing a split or threshold sequence must itself be split or threshold, respectively. Motivated by the fact that threshold graphs and split graphs have characterizations in terms of forbidden induced subgraphs, we define a class $\mathcal{F}$ of graphs to be \textit{dominance monotone} if whenever no realization of $e$ contains an element $\mathcal{F}$ as an induced subgraph, and $d$ majorizes $e$, then no realization of $d$ induces an element of $\F$. We present conditions necessary for a set of graphs to be dominance monotone, and we identify the dominance monotone sets of order at most 3.
\end{abstract}

\section{Introduction} \label{sec: intro}
In this paper, we study the interactions of two aspects of graph degree sequences, namely their relationships under the majorization order, and the induced subgraphs that their realizations may or must not contain.

When partitions of integers with a common sum are ordered via majorization, interesting observations are possible. Here we assume that $d=(d_1,\dots,d_n)$ and $e=(e_1,\dots,e_p)$ are lists of positive integers with their terms in nonincreasing order, and we say that $d$ \emph{majorizes} $e$, denoted $d \succeq e$, if \[\sum_{i=1}^n d_i = \sum_{i=1}^p e_i \quad \text{ and } \quad \sum_{i=1}^k e_i \leq \sum_{i=1}^k d_i \, \text{ for } \, 1 \leq k\leq \min\{p,n\}.\] 

Applying the relation $\succeq$ to all partitions of a fixed positive integer yields a lattice. Brylawski~\cite{Brylawski73} established fundamental properties of this lattice. This lattice serves as a setting for various sandpile models and related chip firing games (see, for example~\cite{Goles02}).

Majorization has a strong connection to graph degree sequences as well. As observed by Ruch and Gutman~\cite{Ruch1979} and others, all 
\emph{graphic} partitions (i.e., degree sequences of simple graphs) among these partitions form an ideal, or downward-closed set, meaning that if $d$ is a degree sequence and $d \succeq e$, then $e$ is a degree sequence as well.

If we restrict our attention to the portion of the majorization poset containing just the graphic partitions, we obtain the \emph{dominance order} on degree sequences having a common sum. Aigner and Triesch~\cite{AignerTriesch} used the dominance order in problems related to the existence of degree sequence realizations having desired properties. Berger~\cite{Berger18} showed that the number of realizations of certain degree sequences is strongly related to majorization relations among the degree sequences. Arikati and Peled studied the majorization gap of a degree sequence~\cite{ArikatiPeled} and showed that degree sequences that lie immediately below the top of the dominance order have Hamiltonian realization graphs~\cite{ArikatiPeled2}. 

The degree sequences near the top of the dominance order belong to interesting graph classes. The maximal degree sequences in the dominance order are known as the \emph{threshold sequences}, and their realizations, the \emph{threshold graphs}, have been shown to have several remarkable properties (see the monograph~\cite{Definitionthreshold} for a survey). Merris~\cite{Merris2002} showed that the more general class of \emph{split graphs}, those whose vertex sets can be partitioned into a clique and an independent set, have degree sequences that are upward-closed in the dominance order, meaning that if $e$ is the degree sequence of some split graph and $d$ is any degree sequence majorizing $e$, then every realization of $d$ is a split graph as well.

In addition to their degree sequence characterizations, the classes of threshold graphs and of split graphs both have characterizations in terms of induced subgraphs. Chv\'atal and Hammer~\cite{Chvatal77} showed that threshold graphs are precisely those graphs that are \emph{$\{2K_2,C_4,P_4\}$-free}, meaning that these graphs have no induced subgraph isomorphic to any of $2K_2$, $C_4$, or $P_4$. F\"{o}ldes and Hammer~\cite{Ham1977} likewise showed that the split graphs are the $\{2K_2,C_4,C_5\}$-free graphs.

Recently \cite{WeaklyThreshold}, the \emph{weakly threshold graphs} were introduced by the first author as those graphs for which the degree sequences satisfied a relaxation of a degree sequence characterization of threshold graphs. Weakly threshold graphs form a subclass of the split graphs, and like the split and threshold graphs, they have a forbidden subgraph characterization and the property that any degree sequence majorizing the degree sequence of a weakly threshold graph is itself the degree sequence of a weakly threshold graph. 

In light of these examples, it appears that we may better understand one facet of the dominance order by considering hereditary graph classes like the threshold, split, and weakly threshold graphs whose degree sequences form upward-closed sets in the dominance order. To do this, we will focus on the corresponding sets of forbidden induced subgraphs. We define a set $\F$ of graphs to be \textit{dominance monotone} if the following property is true:
\begin{quote}
    If $d$ and $e$ are degree sequences such that $d \succeq e$ and every realization of $e$ is $\F$-free, then every realization of $d$ is $\F$-free as well.
\end{quote}
In other words, $\F$ is dominance monotone if the forcibly $\F$-free-graphic sequences form an upward-closed set in each dominance order (precise definitions will be given in the following section).

In this paper we initiate the study of dominance monotone sets, establishing necessary conditions and determining all dominance monotone sets of size at most 3. In Section 2, we recall preliminary notation, definitions, and results on degree sequences, majorization, and forbidden subgraphs. In Section 3 we determine necessary conditions for graphs in dominance monotone sets and use these conditions to determine the dominance monotone sets of order 1. In Sections 4 and 5 we characterize the dominance monotone sets $\mathcal{F}$ for which $|\mathcal{F}|= 2$ and $|\F|=3$, respectively, including the first known dominance monotone examples $\F$ for which the $\F$-free graphs are not a subclass of the split graphs. In Section 6 we present a few concluding remarks and questions.

\section{Preliminaries}
In this section, we recall basic terminology and notions for degree sequences and related concepts. 

All graphs considered here are finite and simple. We use $K_n$, $C_n$, and $P_n$, respectively, to denote the complete graph, the cycle graph, and the path graph having $n$ vertices. We denote the disjoint union of graphs $G$ and $H$ by $G+H$ and the disjoint union of $a$ copies of $G$ by $aG$. We use $G \vee H$ to indicate the join of graphs $G$ and $H$.

We denote the vertex set and edge set of a graph $G$, respectively, by $V(G)$ and $E(G)$, and we define $n(G)=|V(G)|$. We use $\overline{G}$ to denote the complement of $G$. For any $v \in V(G)$, we use $d_G(v)$ to denote the degree of $v$ in $G$, and we write the degree sequence of $G$ as a list $d_G=(d_1, d_2,\dots, d_n)$ having terms in nonincreasing order. At times, particularly when degree sequences appear in pairs, we will write specific degree sequences with small terms without parentheses or commas, as in $d=d_1d_2\cdots d_n$. For multiple identical terms within a degree sequence we may use exponents to indicate multiplicities. We set $\Delta(G)=d_1$ and $\delta(G)=d_n$. Any vertex of $G$ having degree $n(G)-1$ will be called a \emph{dominating vertex}, and any vertex having degree $0$ will be called an \emph{isolated vertex}.

Any graph having such a list $d$ as its degree sequence is called a \textit{realization} of $d$. (Graphs in this paper are unlabeled, meaning that we are not careful to distinguish between isomorphic realizations of a degree sequence).

Turning now to majorization, we use $\D_{2m}$ to denote the dominance order on graphic partitions of $2m$, where $m$ is an integer; it is an elementary result that the sum of the terms in any degree sequence is an even number. We will assume that all terms in elements of $\D_{2m}$ are positive; though of course some graphs do contain isolated vertices, we emphasize that realizations of elements in $\D_{2m}$ are assumed not to.

We may illustrate degree sequences in $\D_{2m}$ and their relationships under majorization using a geometric description known as a \emph{Ferrers diagram}. For $d=(d_1,\dots,d_n) \in \D_{2m}$, define the Ferrers diagram $F(d)$ as a left-justified array made up of $2m$ boxes arranged into rows, with the $i$th row of $F(d)$ consisting of $d_i$ boxes for $i \in \{1,\dots,n\}$. As an illustration, Figure~\ref{fig: Ferrers} displays the Ferrers diagrams of $d= 3221$ and $d'=2222.$ 

\begin{figure}
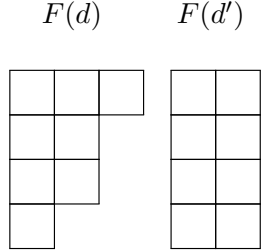

    \centering
$\displaystyle \begin{array}{cc} F(d) \,\,  &  F(d')\,\, \\  \,\, & \,\, \\ \ydiagram{3,2,2,1} &  \ydiagram{2,2,2,2} \end{array}$
    \caption{Ferrers diagrams of $d=3221$ and $d'=2222$.}
    \label{fig: Ferrers}
\end{figure}

We recall a fundamental result on partitions, stating it in terms of Ferrers diagrams.

\begin{lem}[Muirhead's Lemma, \cite{Muirhead03}]
Two degree sequences $d,d' \in \D_{2m}$ satisfy $d \succeq d'$ if and only if $F(d')$ can be obtained from $F(d)$ by moving one or more boxes down to lower rows (even if this process gives rise to new rows) while ensuring that the numbers of boxes in the rows remain in nonincreasing order.
\end{lem}


In Figure~\ref{fig: Ferrers}, moving a box from the first row  to the fourth row of $F(d)$ yields $F(d');$ hence, $3221 \succeq 2222.$ 

We say that a class of elements in a dominance order $\D_{2m}$ is \emph{upward-closed} if whenever $d$ and $e$ are elements of $\D_{2m}$ such that $e$ belongs to the class and  $d\succeq e$, it follows that $d$ belongs to the class as well. For an upward-closed class of degree sequences, Muirhead's Lemma implies that moving any box in the Ferrers diagram of one of these degree sequences to an earlier row produces the Ferrers diagram either of another degree sequence in the class or of a non-graphic partition.

When we consider realizations of degree sequences, it is important to note that a single degree sequence may have multiple nonisomorphic realizations. For this reason, for any graph-theoretic property $\mathcal{P}$ invariant under isomorphism, we say that a degree sequence $d$ is \emph{potentially $\mathcal{P}$-graphic}, or \emph{potentially $\mathcal{P}$}, if at least one of the realizations of $d$ has property $\mathcal{P}$. If every realization of $d$ has property $\mathcal{P}$, we say that $d$ is \emph{forcibly $\mathcal{P}$-graphic}, or \emph{forcibly $\mathcal{P}$}. Thus if $\F$ is a collection of graphs, we say that a degree sequence $d$ is forcibly $\F$-free if no realization of $d$ contains any element of $\F$ as an induced subgraph.

\section{Necessary conditions and dominance monotone singletons}
We work now towards characterizing dominance monotone sets. Recall that a collection $\F$ of graphs is dominance monotone if the class of forcibly $\F$-free sequences is upward-closed in each dominance order $\D_{2m}$.

Since our objective is to identify the dominance monotone sets, we say that a pair $(d, e)$ of degree sequences is a \emph{counterexample pair for $\F$} if $d \succeq e$ and $e$ is forcibly $\F$-free, but $d$ is not, i.e., $d$ has a realization containing an element of $\F$ as an induced subgraph. There is a counterexample pair for $\F$ if and only if $\F$ is not dominance monotone.

For example, the set $\F=\{2K_2, C_4\}$ is not dominance monotone, since the dominance order $\D_{10}$ yields the counterexample pair $(32221, 2^5)$, in which $2^5$ has the chordless 5-cycle (which contains neither $2K_2$ nor $C_4$ as induced subgraphs) as its only realization, and $32221$ has as one of its realizations a chordless 4-cycle with an attached pendant vertex. Since the set $\{2K_2,C_4\}$ is the set of induced subgraphs forbidden for the \emph{pseudo-split graphs}, which further have a degree sequence characterization (see~\cite{Maffray94}), we see that not every hereditary family with a degree sequence characterization forbids a dominance monotone set; more importantly, we also see that dominance monotone sets like $\{2K_2,C_4,C_5\}$ and $\{2K_2,C_4,P_4\}$ may contain non-dominance monotone subsets.

Our first result deals with complements. We use $\overline{G}$ to denote the complement of a graph $G$, and, given a collection $\F$ of graphs, we define $\overline{\F}=\{\overline{F}: F \in \F\}$. 

\begin{thm}\label{thmcomp}
If $\F$ is dominance monotone and no graph in $\mathcal{F}$ has a dominating vertex, then $\overline{\F}=\{\overline{F}: F \in \F\}$ is dominance monotone as well.
\end{thm}

\begin{proof} Assume that $\F$ is dominance monotone and contains no graph with a dominating vertex. Suppose that $e= (e_1, \dots, e_p)$ is forcibly $\overline{\F}$-free and $d \succeq e$, where $d=(d_1, \dots, d_n)$. 

Suppose first that $e_1 < p-1$. Form $\overline{e} = (p-1-e_p, \dots , p-1-e_1)$, the degree sequence of the complement of any realization of $e$, noting that every term of $\overline{e}$ is positive. Muirhead's Lemma implies that $n \leq p$. Now form $\overline{d} = ((p-1)^{p-n},p-1-d_n, \dots, p-1-d_1)$; this is the degree sequence of the graph formed by adding $p-n$ isolated vertices to a realization of $d$ and then taking the complement of the resulting graph. 

Note that $\overline{e}$ is forcibly $\F$-free. Since every term in $\overline{e}$ or $\overline{d}$ is positive, and $\overline{e}$ and $\overline{d}$ are both partitions of $p(p-1)-\sum e_i$, they belong to the same dominance order; furthermore, for each $k \in \{1,\dots,p\},$ 
\begin{multline*}
    \sum_{i=1}^k \overline{e_i} = k(p-1) - \sum_{i=p+1-k}^{p} e_{i} = k(p-1) - \left(p(p-1) - \sum_{i=1}^{p-k} e_{i}\right) = \sum_{i=1}^{p-k} e_{i} - (p-k)(p-1)\\
    \leq  \sum_{i=1}^{p-k} d_{i} - (p-k)(p-1) = \sum_{i=1}^k \overline{d_i}.
\end{multline*}
Hence $\overline{d}$ majorizes $\overline{e}$. Since $\F$ is dominance monotone, $\overline{d}$ is forcibly $\mathcal{F}$-free, and the complement of any of its realizations is forcibly $\overline{\mathcal{F}}$-free. It follows that $d$ is forcibly $\overline{\mathcal{F}}$-free, as claimed.

Suppose now that $e_1=p-1$. Form $\overline{e}' = (p,p-e_p,\dots,p-e_1)$ and $\overline{d}' = (p^{p-n+1},p-d_n,\dots,p-d_1)$; these are precisely the sequences $\overline{e}$ and $\overline{d}$ from the previous paragraph, but with each term increased by one and an extra term of $p$ inserted at the beginning. Each term of $\overline{e}'$ and of $\overline{d}'$ is positive, and similar arguments to those above show that $\overline{d}' \succeq \overline{e}'$. If $\overline{e}'$ is forcibly $\F$-free, then $\overline{d}'$ will be forcibly $\mathcal{F}$-free and hence $d$ will be forcibly $\overline{\F}$-free, as desired, since any realization of $d$ is an induced subgraph of some realization of the complement of $\overline{d}'$. It suffices, then, to note that any realization of $\overline{e}'$ is obtained by adding a dominating vertex to the complement of a realization of $e$. Since no graph in $\F$ has a dominating vertex, if $\overline{e}'$ induces an element of $\F$, the vertices of this induced subgraph must include only vertices not of degree $p$ in $\overline{e}'$. However, the subgraph induced on such vertices is the complement of an $\overline{\F}$-free graph, a contradiction.
\end{proof}

\begin{thm}\label{thmmaxdegree1} In every dominance monotone set, the graph with the lowest number of edges has maximum degree less than or equal to 1.\end{thm}

\begin{proof} Let $\mathcal{F}$ be a dominance monotone set. If 
all graphs in $\F$ with the lowest number of edges have maximum degree greater than 1, then for such a graph $F$, the pair $(d(F), 1^{2|E(F)|})$ is a counterexample pair, since no element of $\F$ is induced in a realization of $1^{2|E(F)|}$, which is a contradiction.\end{proof}

\begin{cor}\label{cormaxmin} 
If $\F$ is a dominance monotone set, then $\F$ contains either a graph with a dominating vertex or a $(|V(F)|-2)$-regular graph $F$; in the latter case $F$ has an even number of vertices.
\end{cor}
\begin{proof} Let $\mathcal{F}$ be a dominance monotone set
in which no graph has a dominating vertex.
By Theorem~\ref{thmcomp},  $\overline{\mathcal{F}}$ is also a dominance monotone set. By Theorem~\ref{thmmaxdegree1}, there exists a graph in $ \overline{\mathcal{F}}$ with maximum degree at most 1. The complement of this graph is in $\mathcal{F}$; call it $F$. Thus, any vertex degree $d$ of $F$ satisfies $d \geq |V(F)|-1-1$. Since $F$ has no dominating vertex, we also have $d \leq \Delta(F) \leq |V(F)|-2$; thus $F$ is $(|V(F)|-2)$-regular. Since the sum of degrees in a graph is always even, $|V(F)|$ must be even.
\end{proof}

Recall from Section 1 that the threshold sequences are the maximal elements of a dominance order, and that their realizations are precisely the $\{2K_2,C_4,P_4\}$-free graphs.

\begin{prop}\label{propthreshold}
If, for every $F \in \{2K_2, C_4, P_4\}$, $\mathcal{F}$ contains an induced subgraph of $F$, then $\F$ is dominance monotone.
\end{prop} 
\begin{proof} Assume that each of $2K_2,C_4,P_4$ has an induced subgraph belonging to $\F$. Every forcibly $\F$-free sequence is then a threshold sequence and is not majorized by any other degree sequence. Thus no counterexample pair exists for $\F$, and $\F$ is dominance monotone.\end{proof}

We can now characterize the dominance monotone sets with size $1$.

\begin{thm}\label{thmsize1} The dominance monotone sets of cardinality $1$ are $\{K_1\}, \{2K_1\},$ and $\{K_2\}.$
\end{thm}
\begin{proof} By Proposition~\ref{propthreshold}, $\{K_1\}, \{K_2\}, \{2K_1\}$ are all dominance monotone sets. Let $\mathcal{F}=\{F\}$ be a dominance monotone set. By Theorem~\ref{thmmaxdegree1}, $\Delta(F) \leq 1$. If $F$ has a dominating vertex then $F$ equals $K_1$ or $K_2$; otherwise, by Corollary~\ref{cormaxmin}, $F$ is $(|V(F)|-2)$-regular. This implies that $|V(F)|-2 \leq 1$, and since $F$ has an even number of vertices, $F$ must be $2K_1$.
\end{proof}

\section{Dominance monotone pairs}
Because graphs with maximum degree at most 1 are necessary elements in dominance monotone sets, by Theorem~\ref{thmmaxdegree1}, we begin this section by establishing a result related to them.

\begin{lem} \label{lem: first construction}
Let $a,b \geq 0$ with $b \geq 3$ if $a=0$ and $b \geq 1$ if $a=1$. If $\F$ is a dominance monotone set containing $aK_2+bK_1$, then $\F$ contains an induced subgraph of a disjoint union of cycles having at most $3a+2b-1$ vertices.
\end{lem}
\begin{proof}
Assume that $a,b \geq 0$ with $b \geq 3$ if $a=0$ and $b \geq 1$ if $a=1$. Assume also that $\F$ is a dominance monotone set containing $aK_2+bK_1$. Consider the degree sequences $d=3^1 2^{3a+2b-3} 1^1$ and $e=2^{3a+2b-1}$ and note that $d$ majorizes $e$. We claim that the degree sequence $d$ is not forcibly $\F$-free. If $a=0$, one realization is the graph obtained by adding the edge $v_1v_{2b-2}$ in a path $v_1v_2\cdots v_{2b-1}$; deleting $v_{2i}$ for all $1 \leq i \leq b-1$ leaves $bK_1$ as an induced subgraph. If $a \geq 1$, one realization is the graph obtained by adding the edge $v_1v_3$ to the path $v_1v_2\cdots v_{3a+2b-1}$; deleting $v_{3i}$ for all $1 \leq i \leq a$ and $v_{3a+2j}$ for $1 \leq j \leq b-1$ (when these vertices exist) leaves $aK_2+bK_1$ as an induced subgraph. 

Every realization of the degree sequence $e$ is a disjoint union of cycles.  Note that if $aK_2+bK_1$ were induced in a disjoint union of cycles on $3a+2b-1$ vertices, we could arrive at such a subgraph by deleting $a+b-1$ vertices; howevever, deleting $a+b-1$ vertices from a disjoint union of cycles leaves an induced subgraph with at most $a+b-1$ components. 

Hence $e$ is forcibly $aK_2+bK_1$-free. Since $d \succeq e$ and $\F$ is a dominance monotone set, some element of $\F$ must be an induced subgraph of some disjoint union of cycles having at most $3a+2b-1$ vertices.
\end{proof}

We now characterize the dominance monotone sets of cardinality 2, as follows.

\begin{thm} \label{thm: pairs}
A set $\F$ of two graphs is dominance monotone if and only if one of the following is true:
\begin{enumerate}
    \item[\textup{(i)}] $\F$ contains one of $K_1$, $2K_1$, or $K_2$;
    \item[\textup{(ii)}] $\F$ is one of $\{K_2+K_1, P_3\}$, $\{K_2+K_1,C_4\}$, or $\{2K_2,P_3\}$.\end{enumerate}
\end{thm}
\begin{proof}
Sufficiency of the conditions (i) and (ii) follows from Proposition~\ref{propthreshold}. We now prove their necessity.

To begin, we show that the only dominance monotone pairs containing $P_3$ or $K_2+K_1$ are the ones indicated, as follows: If $\F=\{P_3,B\}$ is dominance monotone, then since $(211,1111)$ should not be a counterexample pair, $B$ must be induced in $2K_2$; every such graph $B$ yields one of the pairs from Theorem~\ref{thm: pairs}. If instead the dominance monotone set is $\{K_2+K_1,B\}$, then for $(3221,2222)$ to not be a counterexample pair, $B$ must be induced in $C_4$; every possibility for $B$ yields a set from the theorem statement.

Suppose now that $\F=\{A, B\}$ is a dominance monotone set in which $A$ and $B$ each have at least 3 vertices. Further assume that neither $A$ nor $B$ is an induced subgraph of the other; otherwise, if $A$ is induced in $B$, the $\F$-free graphs are precisely the $A$-free graphs, and Theorem~\ref{thmsize1} implies that $\F$ is dominance monotone if and only if the condition (i) holds.

By Theorem~\ref{thmmaxdegree1}, we may assume without loss of generality that $\Delta(A)\leq 1$. Hence $A$ has the form $aK_2+bK_1$ for some nonnegative $a$ and $b$. Since $A$ has at least three vertices, if $a=0$ then $b\geq 3$, and if $a=1$ then $b\geq 1$.

Recall from Corollary~\ref{cormaxmin} that some element of $\F$ either has a dominating vertex or is regular with degree its order minus 2. This element cannot be $A$; otherwise, as in the proof of Theorem~\ref{thmsize1}, $A$ would have two or fewer vertices, contrary to our assumption. Hence $B$ is the element of $\F$ with this property. Lemma~\ref{lem: first construction} implies that $B$ must also be induced in a disjoint union of cycles; thus $\Delta(B) \leq 2$. These several requirements on $B$ imply that it is one of $P_3$, $K_3$, or $C_4$. The case $B=P_3$ was handled previously. If $B$ is $K_3$ or $C_4$, then $(222, 2211)$ or $(32221, 2^5)$, respectively, is a counterexample pair. 
\end{proof}

\section{Dominance monotone triples}

In this section we characterize the dominance monotone sets of cardinality 3. In the following, the \emph{diamond} is the graph $K_4-e$ for an edge $e$.

\begin{thm} \label{thm: triples}
A set $\F$ of three graphs is dominance monotone if and only if one of the following is true:
\begin{enumerate}
    \item[\textup{(i)}] $\F$ contains a dominance monotone singleton or pair;
    \item[\textup{(ii)}] $\F$ is one of $\{2K_2, P_4, \rm{diamond}\}$, $\{K_2+2K_1, P_4, C_4\}$, $\{2K_2,P_4,C_4\}$, $\{2K_2,C_4,C_5\}.$ 
    \end{enumerate}
\end{thm}

The proof will occupy the remainder of this section. We first show the sufficiency of the conditions (i) and (ii). Condition (i) and $\F=\{2K_2,P_4,C_4\}$ both imply that $\F$ is dominance monotone by Proposition~\ref{propthreshold}. That $\{2K_2,C_4,C_5\}$ is dominance monotone was shown by Merris~\cite{Merris2002}. 

\begin{prop}\label{Specialtriples} The triples $\{2K_2, P_4, \rm{diamond}\}$ and $\{K_2+2K_1, P_4, C_4\}$ are dominance monotone.
\end{prop}
\begin{proof}
Since each of the two sets contains complements of the other set's graphs, by Theorem~\ref{thmcomp} it suffices to show that $\F=\{K_2+2K_1, P_4, C_4\}$ is a dominance monotone set.

Assume that $d \succeq e$ and $e$ is forcibly $\F$-free. If $e$ is also forcibly $2K_2$-free, then $e$ is a threshold sequence, and it is vacuously true that $d$ is forcibly $\F$-free. Suppose instead that $e$ is not forcibly $2K_2$-free. 

We claim that any graph that is $\mathcal{F}$-free and contains $2K_2$ as an induced subgraph may have its vertices partitioned into two cliques and a set containing only dominating vertices. Indeed, consider such a graph $G$, and suppose that the edges of some induced subgraph isomorphic to $2K_2$ are $pq$ and $rs$. 

Partition $V(G)$ as $\{p, q\} \cup N(\{p, q\}) \cup R,$ where $N(\{p, q\})$ denotes the set of vertices in $V(G)\setminus \{p,q\}$ that are adjacent to at least one of $p,q$, and $R$ consists of vertices adjacent to neither of $\{p, q\}$. Then $R$ must be a clique. To see that, suppose there exist non-adjacent vertices $x, y \in R.$ Then $\{p, q, x, y\}$ induces $K_2+2K_1.$ Now, if $N(\{p, q\})$ is empty, then the result holds. Therefore, we may assume that $N(\{p,q\})$ is not empty. We claim that each vertex of $N(\{p,q\})$ is adjacent to all vertices in $R$. Suppose that this is not the case. That is, there exists $w \in N(\{p,q\})$ that is not adjacent to all vertices in $R.$ If $w$ has two distinct non-neighbors $z, v$ in $R$, then either $\{z, v, w, p\}$ or $\{z,v,w,q\}$ induces $K_2+2K_1.$ Thus, any vertex $w$ in $N(\{p,q\})$ is adjacent to all vertices but possibly one vertex in $R.$ If $w$ is not adjacent to all vertices in $R$, then denote the non-neighbor of $w$ in $R$ by $t$. Since $w$ has at most one non-neighbor, $w$ must be adjacent to either $r$ or $s$; without loss of generality, suppose it is $r$. Hence $t \neq r$, and then either $\{t, r, w, p\}$ or $\{t,r,w,q\}$ induces a $P_4.$ This contradiction shows that each vertex in $N(\{p,q\})$ is adjacent to all vertices in $R$.

To finish our proof that $G$ may have its vertices partitioned into two cliques and a set containing dominating vertices, it suffices to show that $N(\{p,q\})$ is a clique and that each vertex in $N(\{p, q\})$ to both $p, q$. Suppose that $u,w$ in $N(\{p,q\})$ are not adjacent. If $u,w$ have a common neighbor in $\{p,q\}$, then that neighbor and a common neighbor from $R$, form an induced $C_4$ with $u,w$. Otherwise, $u$ and $w$ have distinct neighbors in $\{p,q\}$, and the subgraph induced by $\{p,q,u,w\}$ is $P_4$, a contradiction.

Lastly, we show that each vertex in $N(\{p,q\})$ is adjacent to both $p$ and $q$. If not, then without loss of generality, there is a vertex $w$ in $N(\{p,q\})$ that is adjacent to $p$ but not to $q$. The subgraph induced by the set $\{r,w,p,q\}$ is then $P_4$, which is a contradiction.

Now suppose that $G$ is a realization of the forcibly $\{K_2+2K_1,P_4,C_4\}$-free sequence $e$. We show that $R = \{r,s\}$. If $R$ contains an additional vertex $x$, then deleting the edges $pq$ and $rs$ and adding the edges $pr$ and $qs$ yields another realization of $e$ in which the vertices $q,r,s,x$ induce a copy of $P_4$, a contradiction, since $e$ was assumed to be forcibly $\{K_2+2K_1,P_4,C_4\}$-free.

Hence $e$ has the general form $e= (k+3)^k (k+1)^4$ where $k$ is a nonnegative integer. Note that the first $k$ terms of $e$ correspond to dominating vertices in $G$ and hence are maximal for the length of this degree sequence. Thus, if $d \succeq e$, then $d$ can only differ from $e$ in the last four terms. It follows from Muirhead's Lemma and inspection that $d$ is the threshold sequence $(k+3)^k (k+2)^1 (k+1)^2 k^1$, which has a unique realization obtained when $k$ dominating vertices are added to $P_3+K_1$. This graph is $\mathcal{F}$-free, so $\{K_2+2K_1,P_4,C_4\}$ is dominance monotone. 
\end{proof}

We now prove the necessity of Conditions (i) and (ii) in Theorem~\ref{thm: triples}. Suppose that $\mathcal{F}= \{ A, B, C\}$ is a dominance monotone set.

If $C$ contains $A$ or $B$ as an induced subgraph, then the $\F$-free graphs are precisely the $\{A,B\}$-free graphs, and $\{A,B\}$ is a dominance monotone pair (possibly containing a dominance monotone singleton), as in (i). Assume henceforth that none of $A,B,C$ is an induced subgraph of another; the order of each of $A, B, C$ is then at least 3.

By Theorem~\ref{thmmaxdegree1}, we assume without loss of generality that $A=aK_2+bK_1$ for some integers $a,b$.  By Corollary~\ref{cormaxmin}, $\F$ contains either a graph with a dominating vertex or a graph that is regular of degree 2 less than its order. As in the previous section we conclude that this graph is not $A$; without loss of generality we assume it is $B$.

By Lemma~\ref{lem: first construction}, $\F$ contains an induced subgraph of a disjoint union of cycles on at most $3a+2b-1$ vertices; we saw there that such a graph cannot contain $aK_2+bK_1$ as an induced subgraph. Hence either $B$ or $C$ is an induced subgraph of a disjoint union of cycles.

If $B$ is induced in a disjoint union of cycles, then $\Delta(B) \leq 2$. Because $|V(B)| \geq 3$ and $B$ has a dominating vertex or is $(|V(B)|-2)$-regular, $B$ must be $P_3$ or $K_3$ or $C_4$. We will handle these possibilities now, along with a few other cases that will be useful in the future.

\begin{lem}\label{lem: P3 K3 C4}
Every dominance monotone triple containing $P_3$ or $K_3$ or $K_2+K_1$ contains a dominance monotone singleton or pair. Every dominance monotone triple containing $C_4$ either contains a dominance monotone singleton or pair or is one of $\{K_2+2K_1,P_4,C_4\},\{2K_2,P_4,C_4\}, \{2K_2,C_4,C_5\}$. Every dominance monotone triple containing $2K_2$ and $P_4$ either contains a dominance monotone singleton or pair or is $\{2K_2,P_4,C_4\}$ or $\{2K_2, P_4, {\rm diamond}\}$.
\end{lem}
\begin{proof}
Let $\F=\{A,B,C\}$ be an arbitrary dominance monotone set. By Theorem~\ref{thmmaxdegree1} we may assume that $A=aK_2+bK_1$ for nonnegative integers. 

If $B=P_3$, then since $(211, 1111)$ is not a counterexample pair, either $A$ or $C$ must be induced in $2K_2$. By Theorem~\ref{thm: pairs} this graph and $B$ then form a dominance monotone pair.

If $A=K_2+K_1$, then since $(3221,2222)$ is not a counterexample pair, either $B$ or $C$ must be induced in $C_4$. By Theorem~\ref{thm: pairs} this graph and $A$ then form a dominance monotone pair.

If $B=K_3$, then since $(3221, 2222)$ is not a counterexample pair, either $A$ or $C$ is an induced subgraph of $C_4$. By Theorem~\ref{thm: pairs} the set $\F$ will contain a dominance monotone singleton or pair 
unless $C=C_4$. With $C=C_4$, since $(32221, 2^5)$ is not a counterexample pair, $\F$ contains an induced subgraph of $C_5$, which must be $A$. Since $A$ has at least three vertices, we conclude that $A=K_2+K_1$; then $\F$ contains the dominance monotone pair $\{K_2+K_1,C_4\}$. 

If $B=C_4$, then since $(32221, 2^5)$ is not a counterexample pair, $\F$ contains an induced subgraph of $C_5$. If $A$ is this subgraph, then either $A$ has fewer than three vertices (in which case $\F$ contains a dominance monotone singleton, satisfying our claim), or $A = K_2+K_1$, which was discussed previously. Assume that $C$ is induced in $C_5$. The cases where $C$ is $P_3$ or $K_2+K_1$ or a graph with fewer than three vertices lead to $\F$ containing a dominance monotone singleton or pair, so we may assume that $C=P_4$ or $C= C_5$.  

If $B=C_4$ and $C=C_5$, then since $(2222,22211)$ is not a counterexample pair for $\F$, the graph $A$ must be induced in $K_2+K_3$ or $P_5$ and hence is induced in $2K_2$. If $A$ is $K_2+K_1$ or has fewer than three vertices, then $\F$ contains a dominance monotone singleton or pair. Otherwise, $A=2K_2$, and $\F=\{2K_2,C_4,C_5\}$. 

If $B=C_4$ and $C=P_4$, then since $(2211, 21111)$ is not a counterexample pair, $A$ is induced in $P_3+K_2$. If $A$ has three or fewer vertices or is $K_2+K_1$, then $\F$ contains a dominance monotone singleton or pair; otherwise, $A$ is one of $3K_1, 2K_2,K_2+2K_1$. If $A=3K_1,$ we have $(43221, 42222)$ as a counterexample pair, a contradiction. When $A=2K_2$ we have  $\mathcal{F}=\{2K_2,P_4,C_4\}$, and when $A=K_2+2K_1$, we have $\F=\{K_2+2K_1,P_4,C_4\}$.

If $A=2K_2$ and $C=P_4$, then consider the pair $(d,e)$, where $d=43322$ (the degree sequence of $K_1 \vee P_4$) and $e=33332$, which has a unique realization that is obtained by subdividing an edge of $K_4$. Observe that the realization of $e$ contains no induced $2K_2$ or $P_4$. Since $\mathcal{F}$ is dominance monotone, $B$ must be induced in the unique realization of $33332$. Since $B$ either is $(|V(B)|-2)$-regular or has a dominating vertex, we see that either $B$ is $C_4$ or $B$ is $P_3$ (which was discussed above) or $K_1 \vee (K_2 + K_1)$ or the diamond graph. The possibility $B=K_1 \vee (K_2 + K_1)$ is eliminated by the counterexample pair $(3221, 2222)$, so we conclude that $\F$ is either $\{2K_2,C_4,P_4\}$ or $\{2K_2,P_4,{\rm diamond}\}$.
\end{proof}

Assume henceforth that the dominance monotone triple $\F$ contains none of $P_3$, $K_2+K_1$, $K_3$, or $C_4$, and that it does not contain the pair $\{2K_2,P_4\}$ as a subset. Having determined the dominance monotone triples where $\Delta(B) \leq 2$, we will assume in the remainder of the proof that $\Delta(B) \geq 3$ and that $C$ is induced in a disjoint union of cycles on at most $3a+2b-1$ vertices.

To help further restrict our search for dominance monotone triples, we present some further requirements for the set $\F$.

\begin{lem} \label{lem: second construction}
If $\F$ is a dominance monotone set containing $aK_2+bK_1$ or $K_1 \vee (aK_2+bK_1)$ for $b \geq 1$ (and $b \geq 3$ if $a=0$), then $\F$ contains an induced subgraph of a graph obtained by subdividing one edge of $K_1 \vee (aK_2+(b-1)K_1)$; any such induced subgraph is $\{aK_2+bK_1\}$-free.
\end{lem}
\begin{proof}
Consider the degree sequences $d=(b+2a)^1 2^{2a} 1^b$ and $e = (b+2a-1)^1 2^{2a+1} 1^{b-1}$. Clearly $d$ majorizes $e$. Observe that the unique realization of $d$ is a graph isomorphic to $K_1 \vee (aK_2+bK_1)$. 

We show that $e$ has at most two realizations. In any realization $G$ of $e$, a vertex of maximum degree has one non-neighbor. If this non-neighbor has degree 1 (which can only happen if $b \geq 2$), then deleting a vertex of maximum degree yields a graph with degree sequence $1^{2a+2}0^{b-2}$, which has a unique realization in $(a+1)K_2+(b-2)K_1$. Thus, $G$ is the graph obtained from $K_1 \vee (aK_2+(b-1)K_1)$ by subdividing a pendant edge, as in the graph on the left in Figure~\ref{fig: pairs}.

If $G$ is a realization of $e$ in which a vertex $v$ of maximum degree has a non-neighbor with degree 2 (which can only happen if $a \geq 1$, since the degree-2 vertex cannot have neighbors among the  vertices of degree 1), then deleting $v$ yields a graph with degree sequence $2^1 1^{2a} 0^{b-1}$, which has a unique realization in $P_3+(a-1)K_2+(b-1)K_1$. Thus $G$ is the graph obtained from $K_1 \vee (aK_2+(b-1)K_1)$ by subdividing an edge of a triangle, as in the graph on the right in Figure~\ref{fig: pairs}.
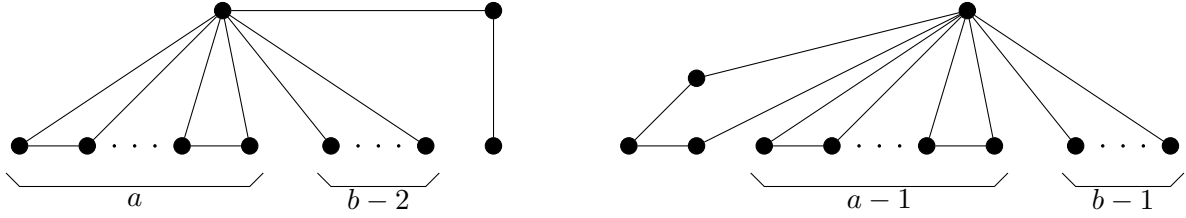
\begin{figure}
    \centering
\begin{tikzpicture}[scale=1.8]
     \node (1) [draw,shape=circle, fill=black, scale=.6] at (0,0) {};
     \node (2) [draw,shape=circle, fill=black, scale=.6] at (0.5,0) {};
     \node (3) [draw,shape=circle, fill=black, scale=.6] at (1.2,0) {};
     \node (4) [draw,shape=circle, fill=black, scale=.6] at (1.7,0) {};
     \node (5) [draw,shape=circle, fill=black, scale=.6] at (2.3,0) {};
     \node (6) [draw,shape=circle, fill=black, scale=.6] at (3,0) {};
     \node (7) [draw,shape=circle, fill=black, scale=.6] at (3.5,0) {};
     \node (8) [draw,shape=circle, fill=black, scale=.6] at (1.5,1) {};
     \node (9) [draw,shape=circle, fill=black, scale=.6] at (3.5,1) {};
    \draw (1)--(2); \draw (3)--(4);
    \draw (1)--(8); \draw (2)--(8); \draw (3)--(8);
    \draw (4)--(8); \draw (5)--(8); \draw (6)--(8);
    \draw (7)--(9); \draw (8)--(9);
    \draw (-0.1,-0.2) -- (0,-0.3) -- (1.7,-0.3) -- (1.8,-0.2);
    \draw (2.2,-0.2) -- (2.3,-0.3) -- (3,-0.3) -- (3.1,-0.2);
    \node at (0.85,-0.4) {$a$};
    \node at (2.65,-0.4) {$b-2$};
    \draw [fill=black] (0.7,0) circle (0.25pt);
    \draw [fill=black] (0.85,0) circle (0.25pt);
    \draw [fill=black] (1,0) circle (0.25pt);
    \draw [fill=black] (2.5,0) circle (0.25pt);
    \draw [fill=black] (2.65,0) circle (0.25pt);
    \draw [fill=black] (2.8,0) circle (0.25pt);
    
    \node (11) [draw,shape=circle, fill=black, scale=.6] at (5.5,0) {};
    \node (12) [draw,shape=circle, fill=black, scale=.6] at (6,0) {};
    \node (13) [draw,shape=circle, fill=black, scale=.6] at (6.7,0) {};
    \node (14) [draw,shape=circle, fill=black, scale=.6] at (7.2,0) {};
    \node (15) [draw,shape=circle, fill=black, scale=.6] at (7.8,0) {};
    \node (16) [draw,shape=circle, fill=black, scale=.6] at (8.5,0) {};
    \node (17) [draw,shape=circle, fill=black, scale=.6] at (4.5,0) {};
    \node (18) [draw,shape=circle, fill=black, scale=.6] at (7,1) {};
    \node (19) [draw,shape=circle, fill=black, scale=.6] at (5,0) {};
    \node (10) [draw,shape=circle, fill=black, scale=.6] at (5,0.5) {};
    \draw (11)--(12); \draw (13)--(14);
    \draw (11)--(18); \draw (12)--(18); \draw (13)--(18);
    \draw (14)--(18); \draw (15)--(18); \draw (16)--(18);
    \draw (17)--(19); \draw (18)--(19);
    \draw (17)--(10); \draw (18)--(10);
    \draw (5.4,-0.2) -- (5.5,-0.3) -- (7.2,-0.3) -- (7.3,-0.2);
    \draw (7.7,-0.2) -- (7.8,-0.3) -- (8.5,-0.3) -- (8.6,-0.2);
    \node at (6.35,-0.4) {$a-1$};
    \node at (8.15,-0.4) {$b-1$};
    \draw [fill=black] (6.2,0) circle (0.25pt);
    \draw [fill=black] (6.35,0) circle (0.25pt);
    \draw [fill=black] (6.5,0) circle (0.25pt);
    \draw [fill=black] (8,0) circle (0.25pt);
    \draw [fill=black] (8.15,0) circle (0.25pt);
    \draw [fill=black] (8.3,0) circle (0.25pt);
\end{tikzpicture}
    \caption{The two realizations of $(b+2a-1)^1 2^{2a+1} 1^{b-1}$.}
    \label{fig: pairs}
\end{figure}

Inspection shows that neither realization of $e$ contains $aK_2+bK_1$ as an induced subgraph, so $e$ is forcibly $\{aK_2+bK_1\}$-free. Since $d \succeq e$ and $\F$ is dominance monotone, $B$ must be induced in one of the graphs in Figure~\ref{fig: pairs}.
\end{proof}

\begin{lem}\label{lem: third construction}
If $\mathcal{F}$ is a dominance monotone set containing $aK_2$ or $K_1 \vee aK_2$, then $\mathcal{F}$ contains an induced subgraph of at least one of the realizations of $\varepsilon=(2a-1)^1 3^1 2^{2a-1}$ (see Figure~\ref{Realization2}); any such induced subgraph is $\{aK_2\}$-free.
\begin{figure}
    \centering
\includegraphics[width=8cm,scale=4]{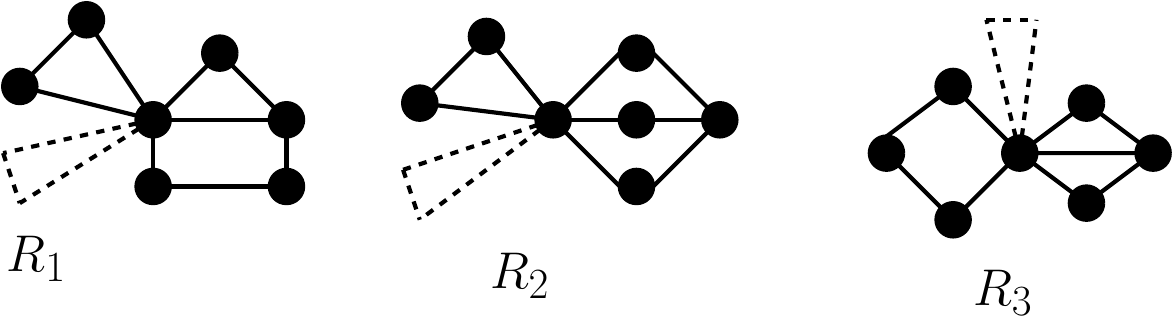}
    \caption{All possible realizations of $\varepsilon=(2a-1)^1 3^1 2^{2a-1}$. In realizations $R_1$ and $R_2$, $a-2$ triangles are attached to a vertex of degree 3 in a realization of $33222$. In realization $R_3$, $a-3$ triangles and a single $C_4$ are attached to a dominating vertex in the diamond.} \label{Realization2}
\end{figure}

If $H$ is a graph that is induced in one of the realizations of $\varepsilon$ and $H$ has a dominating vertex, then $H$ is one of the following:
\begin{itemize}
    \item $K_1 \vee (pK_2 + qK_1)$, where $p \leq a-1$ (and $p=a-1$ only if $q \leq 1$), and $p+q \leq a+1$;
    \item $K_1 \vee (P_3+pK_2+qK_1)$, where $p \leq a-3$ and $p+q \leq a-1$ (this possibility only arises if $a \geq 3$).
\end{itemize}

\noindent If $H$ is an induced subgraph of a realization of $\varepsilon$ with $\Delta(H) \leq 2$, then it satisfies the following:
  \begin{itemize}
      \item if $\Delta(H) \leq 1$, then $H= sK_2+tK_1$, where $s \leq a-1$ (and $s = a-1$ only if $t \leq 1$) and $s+t \leq a+1$. 
    \item if $\Delta(H)= 2$, then $H$ is one of the following:
    \begin{itemize}
        \item $K_3$ or $C_4$;
        \item $P_4+cK_2+dK_1$, where 
        $c+d \leq a-2$;
        \item $P_3+cK_2+dK_1$ where $c+d \leq a-1$ (where $c+d=a-1$ only if $a\geq 3$);
      \item $2P_3+cK_2+dK_1$ where $c+d \leq a-3$ (this possibility only arises if $a \geq 3$).
    \end{itemize}
  \end{itemize}
\end{lem}

\begin{proof}
Given that $\F$ contains $aK_2$ or $K_1 \vee aK_2$, consider the pair $((2a)^1 2^{2a}, \varepsilon)$. Clearly, $d_1 \succeq \varepsilon$, and $d_1$ is not forcibly $\mathcal{F}$-free, since its unique realization is the graph $K_1 \vee aK_2$. Since $\F$ is dominance monotone, this pair of degree sequences is not a counterexample pair, so $\F$ contains an induced subgraph of a realization of $\varepsilon$. 

To see that the induced subgraph is not $aK_2$ or $K_1 \vee aK_2$ when $a \geq 2$, it suffices to realize that the maximum degree vertex in a realization of $\varepsilon$ cannot belong to an induced copy of $aK_2$, for it is adjacent to all but one vertex. Thus an induced copy of $aK_2$ must contain all the other vertices, which is impossible since the degree-3 vertex is adjacent to at least two vertices of degree 2. 
 
In any realization $G$ of $\varepsilon$ the vertex $u$ of maximum degree is adjacent to all but one vertex $v$ of $G$. If $v$ has degree $2$ in $G$, then the graph $G-u$ has degree sequence $2^2 1^{2a-2}$ and hence is isomorphic to either $P_4+(a-2)K_2$ or $2P_3+(a-3)K_2$, and the graph $G$ is of the type shown in realizations $R_1$ or $R_3$ in Figure~\ref{Realization2}. If instead $v$ has degree 3 in $G$, then the degree sequence of $G-u$ is $3^1 1^{2a-1}$ and hence $G-u$ is $K_{1,3} + (a-2)K_2$, leading $G$ to be of the form shown in $R_2$ in Figure~\ref{Realization2}.  

Inspection of the realizations of $\varepsilon$ yields the possibilities for $H$ if $H$ is induced in a realization of $\varepsilon$ and has a dominating vertex or has maximum degree at most 2.
 \end{proof}

With these conditions on $\F$ established, we organize the rest of the proof of the necessity of (i) and (ii) in Theorem~\ref{thm: triples} by recalling that $B$ has a dominating vertex or is $(|V(B)|-2)$-regular. We will handle the two possibilities for the structure of $B$ in separate subsections.


\subsection{Case: $B$ has a dominating vertex}

We begin with two helpful lemmas on dominance monotone sets containing graphs of certain types.

\begin{figure}
    \centering
\includegraphics[width=8cm,scale=4]{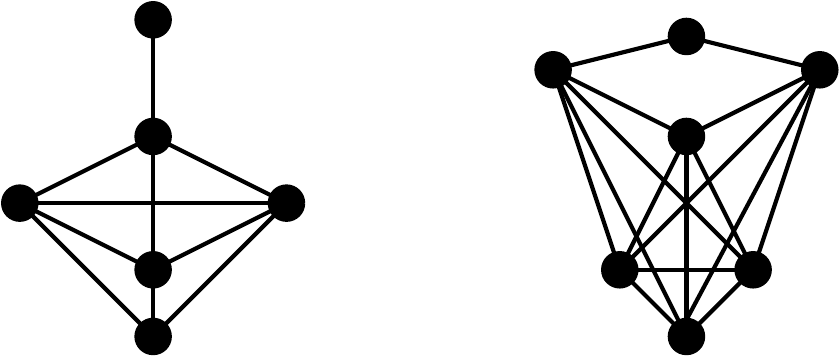}
    \caption{The unique realizations of $4^4 3^1 1^1$ and $5^6 2^1$.}
    \label{Realization3}
\end{figure}

\begin{lem} \label{lem: P4 K6 subdivided}
If $\F$ is a dominance monotone set containing $P_4$, then $\F$ contains an induced subgraph of the graph obtained by subdividing an edge of $K_6$.
\end{lem}
\begin{proof}
Since $\F$ is dominance monotone, the degree sequences $(6^1 5^4 4^1 2^1, 5^6 2^1)$ do not form a counterexample pair; note that a realization of $6^1 5^4 4^1 2^1$ is obtained by adding a dominating vertex to $4^4 3^1 1^1$, which has a realization inducing $P_4$ as shown in Figure~\ref{Realization3}, and the unique realization of $5^6 2^1$ is obtained by subdividing an edge of $K_6$.
\end{proof}

\begin{lem}\label{lem: P3aK2bK1} Let $\mathcal{F}$ be a dominance monotone set containing $P_3+pK_2+qK_1$ or $K_1 \vee (P_3+pK_2+qK_1)$. If $\F$ contains $K_1 \vee P_3$ (i.e., if $q=p=0)$, then $\F$ contain an induced subgraph of $C_5$. If $q=0$ and $p \geq 1$, then $\mathcal{F}$ must contain an induced subgraph $H$ of at least one of the realizations of $e_2= (2p+2)^1 4^1 2^{2p+2}$ in Figure~\ref{Realization17}. 
If $q \geq 1$ then $\mathcal{F}$ must contain an induced subgraph of $K_1 \vee ((p+2)K_2+(q-1)K_1)$.  

If $p+q \geq 1$, then the induced subgraphs described are $\{P_3+pK_2+qK_1\}$-free. Moreover, if $q=0$ and $p\geq 1$, then $H$ satisfies the following:
\begin{itemize}
    \item if $\Delta(H) \leq 1$, then $H=sK_2+tK_1$, where $s\leq p$ and $s+t\leq p+3$. 
    \item if $\Delta(H)=2$, then $H$ is one of the following.
    \begin{itemize}
        \item $K_3$, $C_4$, or $P_4$;
        \item $P_3+sK_2+tK_1$ for some $s,t$ such that $s \leq p-1$ and $s+t \leq p+1$;
        \item $2P_3+sK_2+tK_1$ for some $s,t$ such that $s+t \leq p-2$.
    \end{itemize}
\end{itemize}
\end{lem}
\begin{proof}
When $q=p=0$, the set $\F$ contains $K_1 \vee P_3$, since $\F$ does not contain $P_3$. Since $(3322,22222)$ is not a counterexample pair, $\F$ must contain an induced subgraph of $C_5$.

When $q \neq 0$, it suffices to realize that $(d,e_1)$ is not a counterexample pair, where $d=(2p+q+3)^1 3^1 2^{2p+2} 1^q$ (the degree sequence of $K_1 \vee P_3+pK_2+qK_1$) and $e_1= (2p+q+3)^1 2^{2p+4} 1^{q-1}$ (the degree sequence of $K_1 \vee ((p+2)K_2+(q-1)K_1)$.

When $q=0$ and $p\geq 1$, consider the pair $(d, e_2)$ where $d$ is as above and $e_2=(2p+2)^1 4^1 2^{2p+2}$; since this is not a counterexample pair, $\F$ contains an induced subgraph of a realization of $e_2$. We determine the realizations of $e_2$ as follows. Let $H$ be a realization of $e_2$. Let $u$ and $v$ be the vertices of maximum degree and degree 4, respectively. Observe that $u$ is adjacent to all but one of the other vertices in $H$. If $u$ is not adjacent to $v$, then $H-u$ has degree sequence $4^1 1^{2p+2}$, which is uniquely realized by $K_{1,4}+(p-1)K_2$, and $H$ therefore has the form shown in the first graph in Figure~\ref{Realization17}. If $u$ is adjacent to $v$, then $H-u$ has degree sequence $3^1 2^1 1^{2p+1}$, which has  realizations $T+(p-1)K_2$, where $T$ is the tree obtained by attaching two pendant vertices to an endpoint of $P_3$, and $K_{1,3}+P_3+(p-2)K_2$ (which is possible only if $p \geq 2$). In these cases the graph $H$ has a form shown in the second and third graphs, respectively, in Figure~\ref{Realization17}.

That the realizations of $e_1$ and of $e_2$ are all $\{P_3+pK_2+qK_1\}$-free when $p+q \geq 1$ can be easily verified by inspection. Inspection also confirms the stated conditions on $H$ when $\Delta(H) \leq 2$.
\end{proof}

\begin{figure}
    \centering
\includegraphics[width=8cm,scale=4]{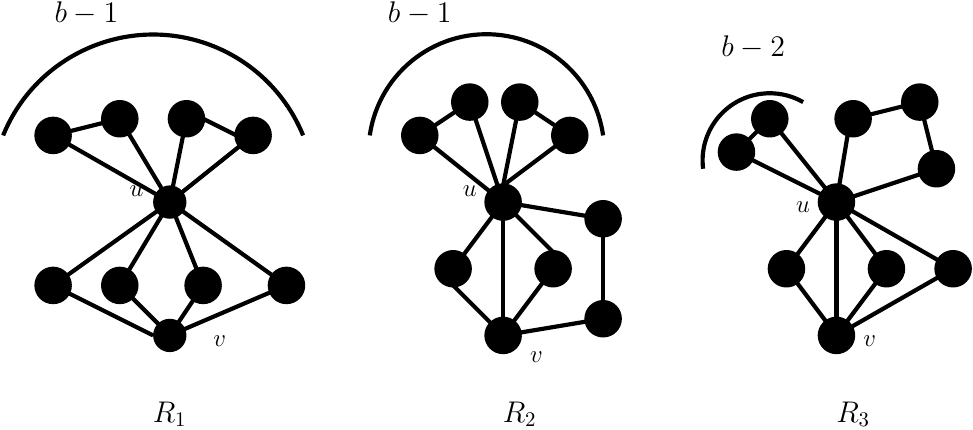}
    \caption{All possible realizations of $e_2= (2p+2)^1 4^1 2^{2p+2}$.}
    \label{Realization17}
\end{figure}


With our preliminary lemmas established, recall that $A=aK_2+bK_1$, that $B$ has a dominating vertex, and that $C$ is induced in a disjoint union of cycles on at most $3a+2b-1$ vertices. We proceed by subcases on the number of isolated vertices in $A$. 

\subsubsection{Subcase 1: $b=0$.} \label{subsub}

Here $A= aK_2$, where $a \geq 2$ by our assumption that $A$ has at least three vertices, and $C$ is induced in a disjoint union of cycles on at most $3a-1$ vertices. By Lemma~\ref{lem: third construction}, $\mathcal{F}$ contains an induced subgraph in at least one of the realizations of $\varepsilon= (2a-1)^1 3^1 2^{2a-1}$, and this graph is not $A$. Therefore, either $B$ or $C$ is induced in at least one realization of $\varepsilon$. 

Suppose first that $B$ is induced in at least one realization of $\varepsilon$. By Lemma~\ref{lem: third construction} $B$ is equal to either $K_1 \vee (pK_2+qK_1)$ or $K_1 \vee (P_3+pK_2+qK_1)$, where $p+q$ is bounded according to the values of $p$ and $q$. We will consider each of these possibilities for $B$ in turn.

\textit{Case: $B= K_1 \vee pK_2$, where $p \leq a-1$}. We may assume that $p \geq 2$, since $B$ is assumed not to be $K_3$. 
Moreover, by Lemma~\ref{lem: third construction}, $\mathcal{F}$ must contain an induced subgraph in at least one of the realizations of $\varepsilon'=(2p-1)^13^12^{2p-1}$; each such realization is $\{pK_2\}$-free and hence $\{A,B\}$-free, so $C$ is induced in some realization of $\varepsilon'$. If $\Delta(C) \leq 1$, then $C= sK_2+tK_1$ for some integers $s,t$ bounded as in Lemma~\ref{lem: third construction}. If $t=0$, then $C$ is induced in $A$, contrary to our assumption, so $C=sK_2+tK_1$ where $ t \geq 1$ and $s \leq p-1$. Thus, by Lemma~\ref{lem: second construction}, $\mathcal{F}$ must contain an induced subgraph of a graph obtained by subdividing one edge of $K_1 \vee (sK_2+(t-1)K_1)$, but none of $A$, $B$, or $C$ is such an induced subgraph, a contradiction.

If $\Delta(C)=2$, Lemma~\ref{lem: third construction} lists all graphs that $C$ can be. Our assumptions exclude the possibilities of $C$ being one $K_3$ or $C_4$. If $C=P_4+s'K_2+t'K_1$, where $s'+t' \leq p-2$, consider the pairs $(d,e_1)$ and $(d,e_2)$, where 
\begin{align*}
d &= (2s'+t'+4)^1 3^2 2^{2s'+2} 1^{t'},\\
e_1 &= (2s'+3)^1 4^1 3^1 2^{2s'+2},\\
e_2 &= (2s'+t'+3)^1 3^2 2^{2s'+3} 1^{t'-1}.
\end{align*}
Note that $d$ is the degree sequence of $K_1 \vee C$. If $t'=0$, then $e_1$ is forcibly $\{A\}$-free because otherwise $2a \leq 2s'+5$, yielding $a-2 \leq s' \leq p-2 \leq a-3$, a contradiction. The sequence $e_1$ is also forcibly $\{B\}$-free since otherwise $2p+1 \leq 2s'+5 \leq 2(p-2)+5$, implying that $B$ is the realization of $e_1$, a contradiction since $B$ has a dominating vertex. Finally, any realization of $e_1$ has exactly one vertex more than $C$; if $s' \geq 1$, then deleting any vertex from such a realization leaves a subgraph with maximum degree at least 3, so $e_1$ is forcibly $\{C\}$-free. Thus $(d,e_1)$ is a counterexample pair if $t'=0$ unless $s'=0$ and hence $C=P_4$. In this case, the result of Lemma~\ref{lem: P4 K6 subdivided} requires that $\F$ contain an induced subgraph of the graph obtained from subdividing an edge of $K_6$. None of $aK_2$, $K_1 \vee pK_2$, or $P_4$ fits this description, which is a contradiction.

If $t'\geq 1$, then $e_2$ is forcibly $\{A\}$-free, since otherwise \[2a \leq 2s'+t'+5 < 2(s'+t')+5 \leq 2(p-2)+5 \leq 2a-1,\] a contradiction. The sequence $e_2$ is forcibly $\{B\}$-free, since otherwise \[ 2p+1 \leq 2s'+t'+5 < 2(s'+t')+5 \leq 2(p-2)+5 = 2p+1,\] a contradiction. The sequence $e_2$ is forcibly $\{C\}$-free because any realization of $e_2$ has exactly one more vertex than $C$, but deleting a single vertex from such a realization cannot leave $t'$ isolated vertices. Thus $(d,e_2)$ is a counterexample pair, and this contradiction concludes the possibility that $C=P_4+s'K_2+t'K_1$.

Suppose instead that, as in Lemma~\ref{lem: third construction}, $C=P_3+s'K_2+t'K_1$, where $s'+t' \leq p-1$ and $s'+t'=p-1$ only if $p \geq 3$. Consider the pairs $(d,e_1)$ and $(d,e_2)$, where \begin{align}\label{eq: P3+sK2+tK1} \begin{split}
d &= (2s'+t'+3)^1 3^1 2^{2s'+2} 1^{t'},\\
e_1 &= (2s'+2)^1 4^1 2^{2s'+2},\\
e_2 &= (2s'+t'+2)^1 3^1 2^{2s'+3} 1^{t'-1}.
\end{split}
\end{align}
The arguments here proceed in much the same way as in the last paragraph, except in the following few ways. To conclude that $e_1$ is forcibly $\{A\}$-free we also note that if $2a \leq 2s'+4$, then $s'=p-1=a-2$, from which it follows that $A$ is a realization of $e_1$, a contradiction. To conclude that $e_1$ is forcibly $\{B\}$-free we note that if $2p+1\leq 2s'+4$, then $B$ can be obtained by deleting one vertex from a realization of $e_1$, and no such vertex deletion yields $B$. To conclude that $e_1$ is forcibly $\{C\}$-free we may assume that $s+t\geq 1$, since by assumption $C \neq P_3$. To conclude that $e_2$ is forcibly $\{B\}$-free, we note that if $2p+1\leq 2s'+t'+4$, then $B$ is a realization of $e_2$, a contradiction.

The above contradictions imply, by Lemma~\ref{lem: third construction}, that $C=2P_3+s'K_2+t'K_1$, where $s'+t' \leq p-3$ and $s'+t'=p-3$ only if $p \geq 3$. Consider the pairs $(d,e_1)$ and $(d,e_2)$, where 
\begin{align*}
d &= (2s'+t'+6)^1 3^2 2^{2s'+4} 1^{t'},\\
e_1 &= (2s'+5)^1 4^1 3^1 2^{2s'+4},\\
e_2 &= (2s'+t'+5)^1 3^2 2^{2s'+5} 1^{t'-1}.
\end{align*}
The arguments showing that $(d,e_1)$ and $(d,e_2)$ are counterexample pairs in the cases $t'=0$ and $t'\geq 1$, respectively, are again analogous to those in the case $C=P_4+s'K_2+t'K_1$ above. We omit the details and conclude that this possibility for $C$ also ends in contradiction.

\textit{Case: $B=K_1 \vee (pK_2+qK_1)$, where $p \leq a-1$ (and $p =a-1$ only if $q \leq 1$) and $p+q \leq a+1$.} By the previous case, we may assume that $q \geq 1$, and since $B$ is not $P_3$, we assume that $q \geq 3$ if $p=0$. 
Then by Lemma~\ref{lem: second construction}, either $A$ or $C$ must also be an induced subgraph of a graph obtained by subdividing an edge of $K_1 \vee (pK_2+(q-1)K_1)$. Since  $A$ is not an induced subgraph, $C$ is, besides being induced in the disjoint union of cycles having at most $3a-1$ vertices. If $\Delta(C) \leq 1$, then $C= sK_2+tK_1$ for some $s,t$ such that $s \leq p+1 \leq a$  and $s+t \leq p+q \leq a+1$. If $t=0$ then $C$ is induced in $A$, and if $s=0$ then $C$ is induced in $B$, contrary to our assumption, so we see that $s, t \neq 0$ and $t \geq 2$ when $s=1$ since $C$ has at least 3 vertices and is not $K_2+K_1$. Thus, by Lemma~\ref{lem: second construction}, $\F$ must contain an induced subgraph of a graph obtained by subdividing an edge of $K_1 \vee (sK_2+(t-1)K_1)$, but such a graph is $\{A,B,C\}$-free, a contradiction. If $\Delta(C) =2$, then since $C$ is not $P_3$ or $K_3$ or $C_4$, we see that $C$ is either $K_3+K_1$ or $P_4$ or $P_3+s'K_2+t'K_1$ where $s'+t' \geq 1$. 
For $C=K_3+K_1$, we find that $(32^31, 2^5)$ is a counterexample pair, a contradiction. When $C=P_4$, we find that $(6^1 5^4 4^1 2^1, 5^6 2^1)$ is a counterexample pair, a contradiction. For $C=P_3+s'K_2+t'K_1$, the degree sequences in~\eqref{eq: P3+sK2+tK1} form counterexample pairs for analogous reasons.







\textit{Case: $B=K_1 \vee (P_3+pK_2+qK_1)$, where $p \leq a-3$ and $p+q \leq a-1$}. As in Lemma~\ref{lem: third construction}, this case requires that $a \geq 3$.

Assume now that $p$ or $q$ is nonzero. By Lemma~\ref{lem: P3aK2bK1}, $\F$ contains an induced subgraph of $K_1 \vee ((p+2)K_2+(q-1)K_1)$ if $q \neq 0$, or an induced subgraph of a realization of $(2p+2)^1 4^1 2^{2p+2}$ if $q=0$. Neither $A$ nor $B$ can satisfy these requirements, so $C$ is the desired induced subgraph, and Lemma~\ref{lem: P3aK2bK1} implies that $C$ is $P_4$ or $sK_2+tK_1$ or $P_3+sK_2+tK_1$ or $2P_3+sK_2+tK_1$ for suitable $s,t$.

If $C=P_4$, then the pair $(2211, 21111)$ is a counterexample pair, a contradiction.

If $C=sK_2+tK_1$ (where Lemma~\ref{lem: P3aK2bK1} tells us $s \leq p+2$), then Lemmas~\ref{lem: second construction} and~\ref{lem: third construction} imply that $\F$ contains an $\{C\}$-free graph $H$ whose largest induced matching has size at most $s+1$. Since $s \leq p+2 \leq a-1$, the graph $H$ does not contain $A$ as an induced subgraph. Since $B$ contains the diamond as an induced subgraph, $H$ is $\{B\}$-free as well unless $H$ is contained in a graph of the form $R_3$ in Figure~\ref{Realization2} having at most $s-3$ triangles, forcing $p \leq s-3 \leq p-1$, a contradiction.

If $C=P_3+sK_2+tK_1$ (where Lemma~\ref{lem: P3aK2bK1} tells us $s \leq p-1$), then Lemma~\ref{lem: P3aK2bK1} implies that $\F$ contains a $\{C\}$-free graph $H$ that is induced in $K_1 \vee ((s+2)K_2 + (t-1)K_1)$ or in a realization of $(2s+2)^1 4^1 2^{2s+2}$. Since all such graphs have largest induced matchings of order at most $s+2$, and $s+2 \leq p+1 \leq a-2$, the graph $H$ is $\{A\}$-free. Since $\F$ is dominance monotone, $H$ must contain $B$ an induced subgraph. Now $K_1 \vee ((s+2)K_2+(t-1)K_1)$ contains no induced $K_1 \vee P_3$, as $B$ does, so $B$ must be induced in a realization of $(2s+2)^1 4^1 2^{2s+2}$. Note that only the realizations $R_2$ and $R_3$ in Figure~\ref{Realization17} contain $K_1 \vee P_3$ as an induced subgraph. Assume that $p+q \geq 1$. The unique vertex of $B$ with degree at least 4 must be the vertex $u$ of maximum degree in $R_2$ or $R_3$, and the unique vertex of degree 3 in $B$ is the vertex of second-highest degree in $R_2$ or $R_3$. In either realization, the remaining vertices adjacent to $u$ do not yield $pK_2+qK_1$ as an induced subgraph, a contradiction, since $B$ is induced in $H$.

If $C=2P_3+sK_2+tK_1$, (where Lemma~\ref{lem: P3aK2bK1} tells us $s+t \leq p-2$), then for $t \neq 0$ we claim that $\mathcal{F}$ must contain a $\{C\}$-free graph $H$ that is induced in $K_1 \vee (P_3 +(s+2)K_2+(t-1)K_1)$; for otherwise $(d, e)$ would be a counterexample pair, where $d= (2s+t+6)^1 3^2 2^{2s+4}1^t$ (the degree sequence of $K_1 \vee (2P_3+2K_2+tK_1)$) and $e= (2s+t+6)^1 3^1 2^{2s+6}1^{t-1}$, since the unique realization of $e$ is $K_1\vee (P_3+(s+2)K_2+(t-1)K_1)$, which contains only one induced $P_3$. It is not hard to see that in this graph the largest induced matching has order at most $s+3 \leq p\leq a-3$, so this graph is also $\{A\}$-free and hence must contain $B$ as an induced subgraph, since $\mathcal{F}$ is dominance monotone. However, since the realization contains exactly one diamond, this leaves only $s+t+1 \leq p-1$ vertices to obtain an induced $pK_2+qK_1$, which is not possible, a contradiction. If $t=0$, then $\mathcal{F}$ must contain a $\{C\}$-free graph $H'$ that is induced in $P_3+(s+2)K_2$, the unique realization of $e'= 21^{2s+6}$, for otherwise $(d', e')$ would be a counterexample pair, where $d'= 2^2 1^{2s+4}$, since the realization of $e'$ contains only one $P_3$. The largest induced matching in $P_3+(s+2)K_2$ is at most $s+3 \leq p+1\leq a-2.$ Thus $A$ is not induced in any realization of $e'$ and since $B$ has maximum degree at least 3, $B$ is not induced either, a contradiction.

\bigskip
The contradictions above show that $B$ is not induced in any realization of $\varepsilon$, so $C$ must be instead. Recall that $\Delta(C) \leq 2$. 

 If $\Delta(C) \leq 1$, then $C= sK_2+tK_1$ and by Lemma~\ref{lem: third construction}, $s+t \leq a+1$, and $s \leq a-2$ if $t>1$; otherwise $s \leq a-1$. If $t=0$, then $C$ is induced in $A$ contrary to our assumption, so $C= sK_2+tK_1$ where $t \geq 1$ and $t \geq 3$ when $s =0$. If both $s$ and $t$ are equal 1, then Lemma~\ref{lem: P3 K3 C4} applies, and $\F$ contains a dominance monotone singleton or pair. In any other case, by Lemma~\ref{lem: second construction} $\mathcal{F}$ must contain a $\{C\}$-free induced subgraph $H$ of a graph obtained by subdividing one edge of $K_1 \vee (sK_2+(t-1)K_1)$. A maximum induced matching in $H$ has at most $s$ edges if $t=1$ and $s+1$ edges if $t>1$. Since $s \leq a-1$, the graph $A$ is not induced in $H$. Then $B$ is induced in $H$, and since $B$ is assumed not to be $P_3$ or $K_3$, we have $B= K_1 \vee (s'K_2+t'K_1)$ for $s',t'$ such that $ s'\leq s$ and $s'+t' \leq s+t$. If $t'= 0$, then  Lemma~\ref{lem: third construction} shows that $\F$ contains a $\{B\}$-free graph $J$ that is induced in a realization of $\varepsilon' = (2s'-1)^1 3^1 2^{2s'-1}$. A maximum induced matching in $J$ has size at most $s'-1 <s<a$, so $J$ is $\{A,C\}$-free, a contradiction. If $t' \neq 0$, by Lemma~\ref{lem: second construction} $\F$ contains a $\{B\}$-free graph $J'$ that is induced in a graph obtained by subdividing an edge of a realization of $K_1 \vee (s'K_2+(t'-1)K_1)$. Again $J'$ is $\{A,C\}$-free, another contradiction.
 
If $\Delta(C) =2$, then by Lemma~\ref{lem: third construction} the graph $C$ is $P_3+sK_2+tK_1$ or $2P_3+sK_2+tK_1$ or $P_4 +sK_2+tK_1$ for suitably bounded values of $s,t$. 

If $C= P_3+sK_2+tK_1$ then $s \leq a-2$ and $s+t \leq a-1$. Consider the pair $(d,e)$, where $d=3^{1}2^{3s+2t+1}1^1$ and $e= 2^{3s+2t+3}$. Since $(d,e)$ is not a counterexample pair, $\mathcal{F}$ must contain an induced subgraph of one of the realizations of $e$, since $d$ has a realization inducing $C$, namely the graph obtained by adding edge $v_1v_4$ to the path $v_1v_2\cdots,v_{3s+2t+3}$. Every realization of $e$ is $\{A\}$-free, since otherwise $3(s+t+1) \leq 3a \leq 3s+2t+3$, which is a contradiction. Realizations of $e$ are also $\{B\}$-free since $\Delta(B) \geq 3$. Finally, deleting $s+t$ vertices from a realization of $e$ leaves at most $s+t-1$ components, so the realization is also $\{C\}$-free, a contradiction.

If $C=2P_3+sK_2+tK_1$, then $s \leq a-3$ and $s+t\leq a-3$. In arguments similar to those of the last paragraph, the set $\mathcal{F}$ must contain an induced subgraph of one of the realizations of $e=2^{3s+2t+7}$, but all such realizations are $\F$-free, a contradiction.

Hence $C= P_4 +sK_2+tK_1$ where $s \leq a-2$ and $s+t \leq a-2$. If both $s$ and $t$ are $0$, then $C=P_4$ and $\mathcal{F}$ must have an induced subgraph of $P_3+K_2$ (otherwise $(2211, 21111)$ is a counterexample pair). Since $\Delta(B) \geq 3$, the induced subgraph is $A$. By our previous assumptions on $A$ we conclude that $A=2K_2$, and by Lemma~\ref{lem: P3 K3 C4} we find $\mathcal{F}= \{2K_2, P_4, \rm{diamond}\}$. Otherwise, $s+t \geq 1$. Thus $\mathcal{F}$ must contain an induced subgraph of one of the realizations of $e=2^{3s+2t+4}$ and we arrive at a contradiction as before in the argument for the case $C=P_3+sK_2+tK_1$.

\subsubsection{Subcase 2: $b \geq 1$.} 

Since $A$ has at least three vertices, and $A$ is not $K_2+K_1$, assume that $a \geq 2$ if $b=1$ and $a\geq 1$ if $b=2$.

By Lemma~\ref{lem: second construction}, $\mathcal{F}$ must contain an induced subgraph of a graph obtained by subdividing an edge of $K_1 \vee (aK_2+(b-1)K_1)$.

If $B$ is induced in an edge-subdivided $K_1 \vee(aK_2+(b-1)K_1)$, then $B= K_1 \vee (pK_2+qK_1)$ for integers $p,q$ such that $p \leq a$ and $p+q \leq a+b-1$. By Lemma~\ref{lem: second construction}, $\mathcal{F}$ contains an induced subgraph $H$ of a graph obtained by subdividing an edge of  $K_1 \vee (pK_2+(q-1)K_1)$. This subgraph of $H$ must be $C$. Hence $C$ is induced in both the disjoint union of cycles having at most $3a+2b-1$ vertices and a graph obtained by subdividing an edge of $K_1 \vee (pK_2+(q-1)K_1)$ where $p \leq a$ and $p+q \leq a+b-1$. 

If $\Delta(C) \leq 1$, then $C= sK_2+tK_1$ for some $s,t$ such that $s \leq p$ and $s+t \leq p+q$. If $t=0$ then $C$ is induced in $A$, contrary to our assumption. A similar contradiction occurs if $s=0$. We assume that $s, t \neq 0$ (and as before, that $C$ is not $K_2+K_1$). By Lemma~\ref{lem: second construction}, $\mathcal{F}$ contains an induced subgraph $H'$ of a graph obtained by subdividing an edge of  $K_1 \vee (sK_2+(t-1)K_1)$, where $s \leq p \leq a$ and $s+t \leq p+q-1\leq a+b-2$. However, $A$ is not induced in any realization of $S'$ and neither is $B$, a contradiction.

If $\Delta(C) =2$, the graph $C$ contains vertex $u$ of maximum degree in $H$. Since $C$ is not $P_3$, $K_3$, or $C_4$, we have $C=P_4$. Since $(2211, 21111)$ is not a counterexample pair, $A = 3K_1$ or $A=K_2+2K_1$. However, when $A$ is $3K_1$ or $K_2+2K_1$ we have respectively $(43221,42222)$ and $(43322,33332)$ as counterexample pairs, another contradiction.

If $B$ is not induced in an edge-subdivided $K_1\vee (aK_2+(b-1)K_1)$, then $C$ must be, in addition to being induced in a disjoint union of cycles having at most $3a+2b-1$ vertices. We again arrive at a contradiction using exactly the same argument as above when $C$ was induced in an edge-subdivided $K_1 \vee (pK_2+(q-1)K_1)$.

\subsection{Case: No graph in $\F$ has a dominating vertex}

Recall that $A=aK_2+bK_1$, where $a,b \geq 0$, and that $\Delta(C) \leq 2$. By Corollary~\ref{cormaxmin}, since $B$ has no dominating vertex, it is $(|V(B)|-2)$-regular and $|B|$ is even. If $|V(B)|=4$ then $B$ is $C_4$, contrary to a previous assumption, so assume that $|V(B)| \geq 6$ and hence $\delta(B) \geq 4$. 

Since no graph in $\F$ has a dominating vertex, Theorem~\ref{thmcomp} implies that $\overline{\F} = \{\overline{A}, \overline{B}, \overline{C}\}$ is dominance monotone. If $b \geq 1$, then $\overline{A}$ has a dominating vertex, so the set $\overline{\F}$ was found in the previous subsection. Assuming that $\mathcal{\F}$ contains no dominance monotone singleton or pair, we conclude that $\overline{\F}$ is equal to $\{2K_2,P_4,\rm{diamond}\}$ and hence $\F=\{K_2+2K_1,C_4,P_4\}$. Suppose henceforth that $b=0$, i.e., that $A=aK_2$ for some $a \geq 2$.

By Lemma~\ref{lem: third construction}, $\mathcal{F}$ contains an induced subgraph of at least one of the realizations of 
$\varepsilon= (2a-1)^1 3^1 2^{2a-1}$, and this induced subgraph is not $A$. Since $\delta(B) \geq 4$, neither is $B$ induced in a realization of $\varepsilon$, and hence $C$ must be. We proceed by considering the cases $\Delta(C) \leq 1$ and $\Delta(C)=2$. 

The statement $\Delta(C) \leq 1$ implies that $C= sK_2+tK_1$, where $s \leq a-1$ (with equality only if $t=0$) and $s+t \leq a+1$ by Lemma~\ref{lem: third construction}. Since we assumed that $C$ is not induced in $A$, we have $t \neq 0$. Then Theorem~\ref{thmcomp} implies that $\overline{\F}$ is dominance monotone, and $\overline{\mathcal{F}}$ contains a graph with a dominating vertex. Thus the set $\overline{\F}$ was found in the previous subsection, where it was shown to be $\{2K_2,P_4,\rm{diamond}\}$; however, this is a contradiction, since $\F$ was assumed to have two graphs with maximum degree at most 1.

If $\Delta(C) = 2$, then by Lemma~\ref{lem: third construction} we have $C$ is $P_3+sK_2+tK_1$ or $2P_3+sK_2+tK_1$ or $P_4 +sK_2+tK_1$ for suitably bounded $s$ and $t$. We may handle these cases using arguments very similar to those at the end of Subsection~\ref{subsub}, noting that though $B$ does not have a dominating vertex, its degrees are high enough for the arguments to work the same way.

\section{Comments and questions}

All of the dominance monotone sets mentioned in Section~\ref{sec: intro} are forbidden subgraph sets for subclasses of the split graphs. The triples $\{2K_2,P_4,\rm{diamond}\}$ and $\{P_4,C_4,K_2+2K_1\}$ from Theorem~\ref{thm: triples}, which respectively do allow $C_4$ or $2K_2$, show that families obtained from forbiddding a dominance monotone set can contain non-split graphs.

We have characterized the dominance monotone sets of size at most 3. Larger dominance monotone sets are also possible; in fact, there are infinitely many and arbitrarily large such sets.

\begin{thm}
Let $t \geq 1$. If $\F_t$ is the set of all graphs of order $t$, and $\F'_t$ is the set of all graphs with exactly $t$ edges, then $\F_t$ and $\F'_t$ are dominance monotone.
\end{thm}
\begin{proof} Take $t \geq 1.$ Assume that $\F_t$ is the set of all graphs of order $t$. Let $d=(d_1, \cdots, d_n)$, $e=(e_1, \cdots, e_p)$ be two degree sequences such that $d \succeq e$ (terms of $d$ and $e$ are assumed to be positive integers). Assume further that $e$ is forcibly $\F_t$-free; that is no realization of $e$ contains an induced subgraph of order $t$. This implies that $p < t.$ From Muirhead's Lemma, we have $n \leq p < t$; thus $d$ must also be forcibly $\F_t$-free. Since $d$ and $e$ were arbitrary, we have our desired result for $\F_t$.

Likewise, if $d$ and $e$ are as above and $e$ is forcibly $\F'_t$-free, then realizations of $e$ have fewer than $t$ edges, so the sum of the terms of $e$ is less than $2t$ by the Handshaking Lemma. Since $d \succeq e$, the sum of terms in $d$ equals the same number, and so every realization of $d$ is $\F'_t$-free as well, establishing our result for $\F'_t$.
\end{proof}

Unfortunately, though larger dominance monotone sets clearly exist, it may be difficult to extend the approach used in this paper to families of larger cardinality without further conditions similar to Theorems~\ref{thmcomp} and~\ref{thmmaxdegree1} and Corollary~\ref{cormaxmin}.

Observe that all known dominance monotone sets $\F$ have the property that $\overline{\F}$ is dominance monotone, even when $\F$ contains a dominating vertex, so we conjecture that the condition in Theorem~\ref{thmcomp} is not necessary: the complements of graphs in any dominance monotone set form a dominance monotone set. The difficulty in proving this lies in the dominance order's degree sequences not containing any 0 terms; it seems difficult to modify the poset to allow 0 terms without undesirable consequences. 

\section*{Acknowledgments}
The authors wish to thank the anonymous referees for their careful reading of the manuscript and their helpful suggestions.

\end{document}